\documentclass{article}
\usepackage{authblk}
\usepackage{amsfonts}
\usepackage{hyperref}

\usepackage[british]{babel}
\usepackage{wrapfig}
\usepackage[comma, square, numbers]{natbib}
\usepackage[utf8]{inputenc}
\usepackage{amssymb}
\usepackage[normalem]{ulem}
\usepackage{amsthm}
\usepackage{graphics}
\usepackage{amsmath}
\usepackage{amsthm}
\usepackage{dirtytalk}
\usepackage{amstext}
\usepackage{subfigure}
\usepackage{engrec}
\usepackage{floatflt}
\usepackage{rotating}
\usepackage[safe,extra]{tipa}
\usepackage[font={footnotesize}]{caption} 
\usepackage{framed}
\usepackage{listings}
\usepackage{pst-node}
\usepackage{tikz-cd} 
\usepackage[titletoc,title]{appendix}

\newcommand{\mc}{\mathcal}
\newcommand{\mb}{\mathbb}

\newcommand{\R}{\mb R}

\newcommand{\N}{\mb N}

\newcommand{\T}{\mb T}

\newcommand{\eea}{\end{align}}

\renewcommand{\epsilon}{\varepsilon}
\renewcommand{\bar}{\overline}
\renewcommand{\tilde}{\widetilde}

\newcommand{\bo}{\boldsymbol}
\renewcommand{\phi}{\varphi}

\DeclareMathOperator{\diverg}{div}
\DeclareMathOperator{\Leb}{Leb}

\DeclareMathOperator{\diam}{diam}

\renewcommand\upsilon{\theta}

\usepackage[normalem]{ulem}

\newtheorem{theorem}{Theorem}[section]

\newtheorem{corollary}[theorem]{Corollary}

\newtheorem{lemma}[theorem]{Lemma}
\newtheorem{proposition}[theorem]{Proposition}
\theoremstyle{definition}
\newtheorem{definition}[theorem]{Definition}
\theoremstyle{definition}

\newtheorem{remark}[theorem]{Remark}
\newtheorem{example}[theorem]{Example}

\newtheoremstyle{question}
{4pt}
{4pt}
{}
{}
{}
{:}
{}
{}

\usepackage{xcolor}

\newcommand{\balgorithm}{\begin{algorithm}\begin{framed}\ }
\newcommand{\ealgorithm}{\end{framed}\end{algorithm}}

\newcommand{\pippo}{\begin{code} \  \begin{lstlisting}[frame=single]}
\newcommand{\pappo}{\end{lstlisting} \end{code}}

\newcommand{\bd}{\begin{definition}}
\newcommand{\ed}{\end{definition}}

\newcommand{\bt}{\begin{theorem}}
\newcommand{\et}{\end{theorem}}
\newcommand{\bp}{\begin{proposition}}
\newcommand{\ep}{\end{proposition}}
\newcommand{\bc}{\begin{corollary}}
\newcommand{\ec}{\end{corollary}} 
\newcommand{\bl}{\begin{lemma}}
\newcommand{\el}{\end{lemma}}
\newcommand{\br}{\begin{remark}}

\newcommand{\er}{\end{remark}}

\DeclareMathOperator{\Lip}{Lip}

\title{Stability of Fixed Points for Nonlinear Selfconsistent Transfer Operators via  Cone Contractions}

\author[1]{Roberto Castorrini }
\author[2]{Stefano Galatolo}
\author[3]{Matteo Tanzi}
\affil[1]{Scuola Normale Superiore, roberto.castorrini@sns.it}
\affil[2]{Dipartimento di Matematica, Università di Pisa, stefano.galatolo@unipi.it}
\affil[3]{Department of Mathematics, King's College London, matteo.tanzi@kcl.ac.uk}

\usepackage{enumitem}

\begin{document}
\maketitle

\begin{abstract}
In this paper we investigate the action of self-consistent transfer operators (STOs) on Birkhoff cones and give sufficient conditions for stability of their fixed points. Our approach relies on the order preservation properties of STOs that can be established via the study of their differential. We focus on the study of STOs arising from strongly coupled maps both deterministic and noisy. Our approach allows for explicit estimates that we use to give examples of STOs with multiple stable fixed points some of which are shown to be far from the asymptotic behaviour of the corresponding system of finite coupled maps and give information only on long transients for the finite dimensional system.
\end{abstract}

\tableofcontents
\section{Introduction}

Self consistent transfer operators (STOs) are nonlinear functions acting on measures that describe the dynamic of the infinite limit of coupled maps interacting via a mean-field (see e.g.\cite{SB} {\cite{bal} \cite{ST} \cite{BLS}  \cite{G}  \cite{BK} and \cite{BUMI} for a review). The stable fixed points of an STO can be interpreted as equilibria of the system in the thermodynamic limit, thus justifying the importance of the study of their existence, stability, and stability under perturbations.  

Most available results investigate fixed points in the case of small coupling. Recently, increasing effort has been put on treating the case where the maps are strongly coupled, see, e.g. \cite{BL} \cite{CGT}. These results propose tools to study stability (or instability) of fixed points via an analysis of the STO's differentiable structure.
In most applications, the STO is only expected to be differentiable (in the Fréchet sense) if seen acting from a space endowed with a certain norm to another space endowed with a weaker norm. This makes it challenging to study iterations of such operators and \cite{CGT} presents one way to overcome these challenges.

In the current paper we put forward a different strategy to the study of stability of fixed points for STOs that uses Birkhoff cones. These are a standard tool in the study of linear transfer operators and decay of correlations of chaotic systems \cite{L}. Their use relies on  a result by G. Birkhoff \cite{B} that, loosely speaking, states that a linear mapping between cones is a contraction with respect to the Hilbert projective metric intrinsically defined on the cones, and provided that the diameter of the image is finite, the contraction is strict. 

However, STOs in presence of strong coupling are highly nonlinear objects. To study their contraction properties on cones we cannot rely on linearity, but instead we leverage their order preservation properties\footnote{A linear application between cones is also order preserving.} which are sufficient to prove contraction -- this is a standard strategy in nonlinear analysis, see for example \cite{A} and \cite{LN}. In turn, we establish order preservation by using a criterion involving the Gateaux differential of the STO. This criterion does not incur in the differentiability issue mentioned above since it does not rely on iterations of the differential; the price we pay is that we need to restrict the domain of the STO to densities with higher regularity.

An advantage of working with cones is that estimates necessary to verify the sufficient conditions for stability of the fixed points are rather explicit. This allows us to give examples of STOs with multiple stable fixed points in presence of strong coupling  where, in contrast, the  weak coupling regime would allow for only one stable fixed point. It also allows us to treat coupled maps with random noise, and give an example of an STO with a stable fixed point that is ``far" from any statistical asymptotic behavior of the finite-dimensional system; this establishes that, as expected, stable fixed points of STOs can in general show up in the finite dimensional system only as long metastable transients observed for times that can be made arbitrarily large increasing the system's size.

As an addendum, we provide a proof that, by coupling maps with noise and allowing their number to go to infinity, we can describe this thermodynamic limit as a STO. This STO is obtained by composing the push-forward of a state-dependent map (as in the noiseless case) with an integral operator whose kernel depends on the noise (see Proposition \ref{Prop:STOnoisyCoup}).

\section{Contraction of nonlinear applications between cones}
In this section we introduce a criterion for contraction of nonlinear order preserving transformation on cones that we will later apply to the STO for coupled maps. 
\begin{definition} A subset  $\mc V$ of a Banach space $V$ is a convex positive cone if the following properties are satisfied: 
\begin{itemize}
\item[i)] for every $\phi\in \mc V$ and $t> 0$, $t\phi\in \mc V$; 
\item[ii)] for every $\phi,\psi\in \mc V$, $\phi+\psi\in \mc V$;
\item[iii)] $\bar{\mc V}\cap (-\bar{\mc V})=\{0\}$.
\end{itemize}
We say that $\mc V$ is \emph{closed} if $\mc V$ is a closed set with respect to the topology on $V$.
\end{definition}
All the cones that we consider from now on will be convex positive cones.
\begin{definition}[Partial ordering on a cone] Given a cone $\mc V$, define the partial ordering
\[
\phi\le_{\mc V}\psi \quad \mbox{iff}\quad \psi-\phi\in \mc V.
\]
\end{definition}

\begin{definition}[Hilbert Metric]\label{Def:HilbMet} Let $\mc V$ be a cone, then for every $\phi,\psi\in \mc V$ define
\[
M_{\mc V}(\phi,\psi):=\inf\{\beta>0:\, \phi\le_{\mc V} \beta \psi\}\quad\quad m_{\mc V}(\phi,\psi):=\sup\{\beta>0:\, \phi\ge_{\mc V} \beta \psi\}
\]
and the Hilbert projective metric
\[
d_{\mc V}(\phi,\psi):=\log\frac{M_{\mc V}(\phi,\psi)}{m_{\mc V}(\phi,\psi)}.
\]
\end{definition}
The function $d_{\mc V}$ is a pseudometric because it only distinguishes between directions.
Below it is going to be useful the following
\begin{definition}
Given a cone $\mc V$ and  $A\subset \mc V$, we define the diameter of $A$ in $\mc V$ as 
\[
\diam_{\mc V}(A)=\sup_{\psi_1,\psi_2\in A\backslash\{0\}} d_{\mc V}(\psi_1,\psi_2).
\]
\end{definition}
Cones of functions play an important role in the study of contraction properties of transfer operators, thanks to the following result of Birkhoff:
\begin{theorem}[\cite{B}]\label{Thm:ContCones}
Assume $\mc V\subset V$ and $\mc V'\subset V'$ are two convex cones with Hilbert metrics $d_{\mc V}$ and $d_{\mc V'}$ respectively. If $\mc P: V\rightarrow  V'$ is a linear transformation such that $\mc P(\mc V)\subset \mc V'$, then
\[
d_{\mc V'}(\mc P\psi_1,\mc P\psi_2)\le [1-e^{-\diam_{\mc V'}(\mc P\mc V)}] d_{\mc V}(\psi_1,\psi_2)
\]
for all $\psi_1,\psi_2\in\mc V$.
\end{theorem}

In this paper, we are going to use a generalization of the above result that substitutes the assumptions of linearity for $\mc P$ with order preservation and homogeneity (as defined below). This generalization is suitable to study the contraction properties of nonlinear STOs.
\begin{definition}
A mapping $\mc T:\mc V\rightarrow \mc V'$ between two cones is said to be \emph{homogeneous} if $\mc T(\lambda v)=\lambda \mc T(v)$ for every $\lambda>0$ and $v\in \mc V$.
\end{definition}

An analogue of the following proposition can be found for example in \cite{LN}.
\begin{proposition}\label{Thm:ContractionProperties}
Let $\mc V$, $\mc V'$ be  cones  with $\mc V'\subset\mc V$ and let $U\subset \mc V$ be a convex cone. Then if $\mc T:\mc V\rightarrow \mc V'$ is  homogeneous  and  $\mc T|_U:(U,\le_{\mc V})\rightarrow (\mc V',\le_{\mc V'})$ is order preserving, then for every $\phi,\psi\in U$
\[
d_{\mc V}(\mc T\phi,\mc T\psi)\le [1-e^{-\diam_{\mc V}(\mc V')}] d_{\mc V}(\phi,\psi).
\]
\end{proposition}
\begin{proof}
Since $\mc T$ is order preserving in $\mc V$ when restricted to $U$, for every $\phi,\psi\in U$ s.t. $\phi\le_{\mc V} \psi$, one has $\mc T(\phi)\le_{\mc V'} \mc T(\psi)$. Combining this with homogeneity, implies that for every $\beta>0$ and $\phi,\psi\in U$, $\beta\psi\in U$ and if $\phi\le_{\mc V} \beta \psi$ then
\[
\mc T(\phi)\le_{\mc V'} \mc T(\beta \psi)= \beta\mc T(\psi),
\]
and therefore
\[
m_{\mc V'}(\mc T(\phi),\mc T(\psi))\ge m_{\mc V}(\phi,\psi).
\]
Analogously $M_{\mc V'}(\mc T(\phi),\mc T(\psi))\le M_{\mc V}(\phi,\psi)$, and 
\[
d_{\mc V'}(\mc T(\phi),\mc T(\psi))\le d_{\mc V}(\phi,\psi).
\]       

\noindent Since $\mc V'\subset \mc V$,  Theorem \ref{Thm:ContCones} can be applied to the linear inclusion map $\iota:\mc V'\rightarrow \mc V$
\[
d_{\mc V}(\mc T\phi,\mc T\psi)\le [1-e^{-\diam_{\mc V}(\mc V')}] d_{\mc V'}(\mc T\phi,\mc T\psi)
\] and the result follows. 
\end{proof}

 To establish order preservation of a mapping between cones, it becomes handy the following criterion  (a variation of  Theorem 1.3.1 from \cite{LN}) for order preservation of a map in terms of its Gateaux differential.
\begin{proposition}\label{Thm:OrderPreservationCondition}
Consider $(Y,\|\cdot\|_Y)$ a Banach space,  $\mc V\subset Y$  a convex positive cone closed in $Y$, $U\subset \mc V$ a convex subcone, and $\mc T:\mc V\rightarrow\mc V$ a map that restricted to $U$ is  Gateaux differentiable 
with respect to the convergence in $(Y,\|\cdot\|_Y$)\footnote{More precisely, for every $\phi\in U$ and $\psi\in \mc V$, we call $ D \mc T_\phi$ the Gateaux derivative of $\mc T$ at $\phi$ evaluated at $\psi$ the limit in $(Y,\|\cdot\|_Y)$ 
\[
\lim_{h\rightarrow 0}\frac{\mc T(\phi+h\psi)-\mc T(\phi)}{h}.
\]} and such that for every $\phi,\psi\in U$ and every $\xi\in \mc V$, $[0,1]\ni t\mapsto D\mc T_{t\phi+(1-t)\psi}(\xi)$ is continuous w.r.t. $\|\cdot\|_Y$.  

If $D\mc T_\phi(\mc V)\subset \mc V$ for all $\phi\in U$, then  $\mc T|_{U}: (U,\le_{\mc V})\rightarrow (\mc V,\le_{\mc V})$ is order preserving.
\end{proposition}
\begin{proof}
Let $\phi,\psi\in U$ such that $\phi\le_{\mc V} \psi$. Then, by convexity of $U$, $\mc T$ is continuously Gateaux differentiable at $t\psi+(1-t)\phi\in U$ for every $t\in[0,1]$. This together with $\phi\le_{\mc V} \psi$  implies that the map $t\mapsto \mc T(t\psi+(1-t)\phi) $ is differentiable, in fact
\begin{align*}
 &\lim_{h\rightarrow 0}\frac{\mc T((t+h)\psi+(1-t-h)\phi)-\mc T(t\psi+(1-t)\phi)}{h}=\\
 &\quad =\lim_{h\rightarrow 0}\frac{\mc T(t\psi+(1-t)\phi+h(\psi-\phi))-\mc T(t\psi+(1-t)\phi)}{h}\\
 &\quad =D\mc T_{t\psi+(1-t)\phi}(\psi-\phi),
\end{align*}
where we used that $\phi\le_{\mc V} \psi$ implies that $\psi-\phi\in \mc V$.
From the above and the fundamental theorem of calculus  
\begin{align*}
\mc T(\psi)-\mc T(\phi)&=\int_0^1\frac{d}{dt}\mc T(t\psi+(1-t)\phi) dt\\
&=\int_0^1D\mc T_{t\psi+(1-t)\phi}(\psi-\phi) dt.
\end{align*}
By assumption, $D\mc T_{t\psi+(1-t)\phi}(\psi-\phi)\in \mc V$ for all $t\in[0,1]$, and the convexity of $\mc V$ implies that $\mc T(\psi)-\mc T(\phi)\in \mc V$.
\end{proof}

\section{Differential of STOs Arising in Coupled Maps}
In this section we introduce the main class of STOs that we are interested in.



\begin{definition}\label{Def:Self-ConsOp} Given  $f\in C^0(\T^n, \T^n)$, $H\in C^0(\T^n\times\T^n, \R^n)$,  $\delta\in\R$, and $\mu\in \mc M_1(\T^n)$, where $\mc M_1(\T^n)$ denotes the set of Borel probability measures on $\T^n$, define
\begin{itemize}
\item $P:\mc M_1(\T^n)\rightarrow \mc M_1(\T^n)$ to be  the transfer operator for $f$\footnote{The transfer operator of a measurable map $f$ is defined as
\[
f_*\mu(A)=\mu(f^{-1}(A))
\]
for any $A$ measurable.};
\item $g_{\mu}:\T^n\rightarrow \T^n$
\[
g_{\mu}(x)=x+\delta\int_{\T^n}H(x,y)d\mu(y)\mod 1
\]
which is the mean-field coupling map when a (finite or infinite) system of coupled maps is in the state $\mu$;
\item $ L_\mu:\mc M_1(\T^n)\rightarrow \mc M_1(\T^n)$  the (linear) transfer operator of $g_\mu$;
\item $\mc L: \mc M_1(\T^n)\rightarrow \mc M_1(\T^n)$, the STO for the coupling  defined as
\[
\mc L\mu:=L_\mu(\mu);
\]
\item $\mc T:\mc M_1(\T^n)\rightarrow \mc M_1(\T^n)$, the STO for the whole system 
\[
\mc T:=P\mc L.
\]
\end{itemize}
\end{definition}
It can be argued (see e.g. \cite{G} and \cite{ST}) that the STO $\mc T$ in the above definition describes the thermodynamic limit for $N\rightarrow \infty$ of the system of coupled maps $F:(\T^n)^N\rightarrow (\T^n)^N$ defined as:
\begin{equation}\label{Eq:DetCoupledMaps}
x_i(t+1)=F_i(x_1(t),...,x_N(t)):=f\left(x_{i}(t)+\delta \sum_{j=1}^NH(x_i(t),x_j(t)) \mod 1\right)
\end{equation}
where $x_i(t)\in \T^n$ and $i=1,...,N$.

The goal in what follows is to investigate the stability of the fixed points of $\mc T$, i.e.  the observable equilibria in the thermodynamic limit, by studying its action on convex cones of functions\footnote{The cone depends on the properties of the dynamics. For example, when dealing with coupled uniformly expanding maps, we will consider certain cones of $\log$-Lipschitz functions.}. This will lead us to   conditions on the coupling strength (size of $\delta$) that are sufficient  to ensure  stability of fixed points.


%

From now on, we work with $f\in C^3(\T^n,\T^n)$, $H\in C^3(\T^n\times\T^n,\R)$, and measures having density with respect to Lebesgue in, at least, $C^1(\T^n,\R^+)$.

First of all let's extend the  definition of  $g_\mu$   to any positive measure $\mu$ with $C^1$ density.
\begin{definition}\label{Eq:Defg}
For every  $\phi\in C^1(\T^n,\R^+)$, let $g_\phi\in C^3(\T^n,\T^n)$ be defined as
\begin{equation}\label{Eq:Expressgphi}
g_\phi(x):=x+\delta\int_{\T^n}H(x,y)\frac{\phi(y)}{\int \phi}dy\mod 1.
\end{equation}
$L_\phi$, $\mc L$, $\mc T$ are defined analogously as above, and $\int \phi$ is used as short-hand notation for the integral with respect to the Lebesgue measure.
\end{definition}

With this choice, the STOs $\mc L$ and $\mc T$ are homogeneous: $\mc L(\lambda\phi)=\lambda\mc L(\phi)$ for all $\lambda>0$ and $\phi\in C^1(\T^n,\R^+)$ allowing us to use the criterion in Proposition \ref{Thm:ContractionProperties}.

\begin{lemma}\label{Lem:Diffg}
The application $\phi\mapsto g_\phi$ defined from $C^1(\T^n,\R^+)$ to $C^3(\T^n,\T^n)$  is Gateaux differentiable (in $C^0$),  its derivative is
\begin{equation}\label{Eq:Diffg}
Dg_\phi(\psi)(x):=\lim_{s\rightarrow 0}\frac{g_{\phi+s \psi}-g_\phi }{s}(x)=\delta \frac{1}{\int \phi}\int_{\T^n}H(x,y)\left[\psi-\phi\,\frac{\int\psi}{\int\phi} \right](y)dy
\end{equation}
for every $\phi\in C^1(\T^n,\R^+)$ and $\psi\in C^1(\T^n,\R)$, and $x\mapsto Dg_\phi(\psi)(x)$ is a $C^3$ function.
\end{lemma}
\begin{proof}
For $|s|$ small enough,
\begin{align*}
g_{\phi+s \psi}(x)-g_\phi(x)&=\delta\int_{\T^n}H(x,y)\left[\frac{(\phi+s \psi)(y)}{\int \phi+s \psi}-\frac{\phi(y)}{\int \phi}\right]dy\\
&=\delta\int_{\T^n}H(x,y)\left[\frac{(\phi+s \psi)(y)}{\int \phi+s \psi}-\frac{\phi(y)}{\int \phi+s \psi}\right]dy+\\
&\quad\quad+\delta\int_{\T^n}H(x,y)\left[\frac{\phi(y)}{\int \phi+s \psi}-\frac{\phi(y)}{\int \phi}\right]dy\\
&=\delta s\int_{\T^n}H(x,y)\left[\frac{\psi}{\int\phi+s\psi}-\frac{\phi\int\psi}{(\int\phi)(\int\phi+s\psi)} \right](y)dy
\end{align*}
from which we obtain \eqref{Eq:Diffg}, and it is easy to check from this expression that $x\mapsto Dg_\phi(\psi)(x)$ is $C^3$, given that $H\in C^3(\T^n\times\T^n,\R)$.
\end{proof}

The following lemma, whose simple proof is omitted, will be useful in what follows. We will denote by $\partial_1H$ the partial derivative of $H$ with respect to the first coordinate.
\begin{lemma}\label{Lem:BOundsGphi}
For $g_\phi$ defined as above
\begin{equation}
|g_\phi'|_\infty\le 1+\delta |\partial_1H|_\infty,\quad
|g_\phi''|_\infty\le \delta |\partial_1^2H|_\infty,\quad
|g_\phi'''|_\infty\le \delta |\partial_1^3H|_\infty.
\end{equation}
Furthermore, for $|\delta|<|\partial_1H|_\infty^{-1}$, $g_\phi$ is a diffeomorphism and for any $\bar \delta\in(0,|\partial_1H|_\infty^{-1})$ there is a constant $K$ depending on $|\partial_1H|_\infty$ and $\bar\delta$  such that
\[
|(g_\phi^{-1})'|_\infty\le 1+K\delta
\]
for every $\delta$ with $|\delta|<\bar\delta$.
\end{lemma}

We now obtain an expression for the Gateaux derivative  $D\mc T_{\phi}(\psi)$ in terms of $Dg_\phi(\psi)(x)$. We only care about the differential at $\phi\in C^2(\T,\R^+)$ evaluated along the direction of $\psi\in C^1(\T,\R^+)$, since the cones we consider are subsets of these spaces. 

\begin{proposition}\label{Lem:ExpressionforDifferential}
 For every $\phi\in C^2(\T^n,\R^+)$ and $\psi\in C^1(\T^n,\R^+)$, assuming that $|\delta|$ is small enough so that $g_\phi$ is a diffeomorphism, the Gateaux derivative (in $C^0$) of $\mc T$ at $\phi$ evaluated on $\psi$ is
\[
D\mc T_\phi(\psi)=\lim_{s\rightarrow 0}\frac{\mc T(\phi+s\psi)-\mc T(\phi)}{s}=P\,D\mc L_{\phi}(\psi)
\]
where
\begin{equation}\label{Eq:DiffmcL}
D\mc L_\phi(\psi)=\lim_{s\rightarrow 0}\frac{\mc L(\phi+s\psi)-\mc L(\phi)}{s}=L_\phi(\psi)-\diverg\left(L_\phi[\phi \,Dg_{\phi}(\psi)]\right)
\end{equation}
and the convergence of the limits is in $(C^0(\T^n,\R),\|\cdot \|_{C^0})$.
\end{proposition}
\begin{proof}
Without loss of generality, let us assume that the integrals of $\phi$ and $\psi$ are equal to one.

\begin{eqnarray*}
L_{(\phi+\epsilon \psi)}(\phi+\epsilon \psi) &=&L_{(\phi+\epsilon
\psi)}(\phi)+\epsilon L_{ (\phi+\epsilon \psi)}(\psi)  \label{1} \\
&=&L_{ \phi}(\phi)+\epsilon \frac{\lbrack L_{ (\phi+\epsilon
\psi)}-L_{ \phi}]}{\epsilon }(\phi)  \notag \\
&&+\epsilon L_{ (\phi+\epsilon \psi)}(\psi)  \notag
\end{eqnarray*}
and from the above follows that the Gateaux differential of $\mc L$ is given by
\begin{align}
D\mc L_\phi(\psi)&=\lim_{\epsilon\rightarrow 0}\frac{L_{(\phi+\epsilon \psi)}(\phi+\epsilon \psi)-L_\phi(\phi)}{\epsilon}\nonumber\\
&=\lim_{\epsilon\rightarrow 0} \frac{\lbrack L_{ (\phi+\epsilon\psi)}-L_{ \phi}]}{\epsilon }(\phi)+L_{\phi}\psi.  \label{Eq:ExpDiff}
\end{align}

Now we evaluate the  limit in \eqref{Eq:ExpDiff}, and   establish its convergence in $C^0$.
Take any function $h\in C^{\infty}(\T^n,\R)$, and recall that by Lemma \ref{Lem:Diffg} $g_\phi$ is Gateaux differentiable at $\phi$. Then, by the mean value theorem, there is $\xi$ -- depending on $\epsilon$ -- with $|\xi|<\epsilon$ such that
\begin{eqnarray*}
\langle[L_{(\phi+\epsilon \psi)}-L_{\phi}](\phi),h\rangle &=&\langle \phi,h\circ g_{\phi+\epsilon \psi}-h\circ g_{\phi}\rangle \\
&=&\left\langle \phi,\,\epsilon\frac{d}{ds}(h\circ g_{\phi+s\psi})|_{s=\xi} \right\rangle
\end{eqnarray*}
where $\langle h_1,h_2\rangle$ denotes Lebesgue integration of the product $h_1h_2$.
The chain rule gives
\[
\frac{d}{ds}(h\circ g_{\phi+s\psi})|_{s=\xi}(x)=Dh_{g_{\phi+\xi\psi}(x)}Dg_{\phi+\xi\psi}(\psi)(x)
\]
so 
\[
\langle[L_{(\phi+\epsilon \psi)}-L_{\phi}](\phi),h\rangle=\epsilon\left\langle \phi,\,Dh_{g_{\phi+\xi\psi}(x)}Dg_{\phi+\xi\psi}(\psi)(x)\right\rangle.
\]
Define the vector field $v_\xi(x):=Dg_{\phi+\xi\psi}(\psi)(x)$ and denote by $v_{\xi,i}$ its components. From the smoothness assumption and  Lemma \ref{Lem:Diffg}, this is a $C^3$ vector field. Changing coordinates, and integrating by parts
\begin{align*}
\left\langle \phi,\,Dh_{g_{\phi+\xi\psi}}v_\xi \right\rangle&=\int_{\T^n} \phi(x) \sum_{i=1}^n\partial_ih(g_{\phi+\xi\psi}(x)) v_{\xi,i}(x)dx\\
&=\sum_{i=1}^n\int_{\T^n}L_{\phi+\xi\psi}(\phi v_{\xi,i}(x))\partial_ih(x)dx\\
& =-\sum_{i=1}^n\int_{\T^n}\partial_i\left[L_{\phi+\xi\psi}(\phi v_{\xi,i})\right](x)\cdot h(x)dx\\
&=-\left\langle\sum_{i=1}^n\partial_i\left[L_{\phi+\xi\psi}(\phi v_{\xi,i})\right],h\right\rangle.
\end{align*}

Since $h\in C^\infty(\T^n,\R)$ is arbitrary 
\[
-\sum_{i=1}^n\partial_i\left[L_{\phi+\xi\psi}(\phi v_{\xi,i})\right]=\frac{1}{\epsilon}[L_{(\phi+\epsilon \psi)}-L_{\phi}](\phi)
\]
and
\[
\lim_{\epsilon\rightarrow 0}\frac{1}{\epsilon}[L_{(\phi+\epsilon \psi)}-L_{\phi}](\phi)=-\sum_{i=1}^n\partial_i\left[L_{\phi+0}(\phi v_{0,i})\right]= -\diverg\left(L_\phi[\phi Dg_{\phi}\psi]\right)
\]
pointwise.
If $|\epsilon|$ is assumed to be bounded by some $\epsilon_0>0$, $\frac{1}{\epsilon}[L_{(\phi+\epsilon \psi)}-L_{\phi}](\phi)$ are in a bounded ball of $C^1$, and by Ascoli-Arzel\'a the convergence of the limit above is  in $C^0$. 
\end{proof}

\section{Uniformly Expanding Coupled Maps}
Once the above framework is in place, one can find conditions for the STOs arising from the coupled uniformly expanding maps to have a stable fixed point. A similar result has been obtained in (\cite{CGT}), but the proof we give in this new framework allows to obtain explicit estimates of the coupling strength for which the STO admits a stable fixed point. To simplify the notation, we are going to restrict ourselves to the case of 1D maps, but similar results are carried out in the multidimensional case with no conceptual differences.

\subsection{Stable fixed point for an STO arising  in uniformly expanding coupled maps}
A differentiable map $f:\T\rightarrow\T$ is uniformly expanding if there is $\sigma>1$ such that $|f'(x)|>\sigma$. Consider  cones of $C^1$ and $a$-$\log$-Lipschitz functions 
\[
\mc V_a:=\left\{\phi\in C^1(\T,\R^+):\, \frac{\phi(x)}{\phi(y)}\le \exp\left(a|x-y|\right)\right\}
\]
where $|x-y|$  denotes the Euclidean distance on $\T$ between $x$ and $y$\footnote{Equivalently, one can define the cone above by requiring that $\left|\frac{d\log \phi(x)}{dx}\right|\le a$.}.
If $P$ is the transfer operator of a $C^2$ uniformly expanding map $f$ of $\T$, it is well known (e.g. \cite{L}) that  there are $a_0>0$ and $\lambda\in(0,1)$ such that 
\begin{equation}\label{Eq:ContP}
P(\mc V_a)\subset\mc V_{\lambda a}
\end{equation}
for every $a>a_0$, and that
\begin{equation}\label{Eq:DiamLogLipFun}
\diam_{\mc V_{a}}(\mc V_{\lambda a})<\infty.
\end{equation}

For $\alpha\ge 0$, let's  define convex cones
\[
\mc C_\alpha:=\left\{\phi\in C^2(\T,\R^+):\,\left|\frac{\phi''(x)}{\phi(x)}\right|\le \alpha\right\} \quad\mbox{and}\quad U_{a,\alpha}:= \mc V_a\cap\mc C_\alpha.
\]
One can show that if the coupling strength is sufficiently small, for suitable values of the parameters $a$ and $\alpha$, the STO $\mc T$ has an attractive fixed point in $U_{a,\alpha}$. Let us remark that we cannot just work on $\mc V_a$ and we need to restrict ourselves to $\mc C_\alpha$ in order to control the derivative of $\mc T$. 
The main result of this section is the following.
\begin{theorem}\label{Thm:UnifExpCoupMaps}
Given $\mc T$ an STO  as in Definition \ref{Def:Self-ConsOp} with $f\in C^3(\T,\T)$ a uniformly expanding map, and $H\in C^3(\T\times\T,\R)$, there exists $\tilde\delta>0$ such that if the coupling strength $\delta$ satisfies $|\delta|\le \tilde\delta$, then $\mc T$ has an attracting fixed point in $U_{a,\alpha}$ for some $a,\alpha>0$.
\end{theorem}
The above is an application of Proposition \ref{Thm:OrderPreservationCondition}. In the following subsection we are going to show  how to use the framework below to obtain explicit estimates on $\tilde \delta$ in  specific instances. 

There are two main steps to prove this result: 1) the first is to show that $\mc T$ keeps $U_{a,\alpha}$ invariant for a suitable choice of $a$ and $\alpha$; 2) the second  is  to prove order preservation by studying the Gateaux derivative of $\mc T$ restricted to $U_{a,\alpha}$. We start by studying the invariance of $U_{a,\alpha}$.

\begin{lemma}\label{Lem:ConeCoupling}
For every $a_1>0$ there is $C_1>0$ -- depending on $a_1$ -- such that for every $\delta$ satisfying  $|\delta|<\bar \delta$,  with $\bar\delta$ as in Lemma \ref{Lem:BOundsGphi}, and every $a>a_1$
\begin{equation}\label{Eq:ConeCond}
 L_\phi(\psi)\in \mc V_{[1+C_1\delta]a}\quad\quad \forall\phi,\psi\in\mc V_a.
\end{equation}
\end{lemma}
\begin{proof}
Fix $a_1>0$ and pick any $a>a_1$. If $|\delta|<\bar\delta$, $g_\phi$ is a diffeomorphism $\forall \phi\in \mc V_a$ and 
\[
L_\phi(\psi)=\frac{\psi}{|g_\phi'|}\circ g_\phi^{-1}.
\]
It is immediate to verify that if $\psi$ is $C^1$, so is $ L_\phi(\psi)$ and 
\begin{align}
\frac{( L_\phi(\psi))'}{L_\phi(\psi)}&=\frac{|g_\phi'|}{\psi}\circ g_\phi^{-1}\left(\frac{\psi'}{|g_\phi'|^2}\circ g_\phi^{-1}-\frac{\psi g_\phi''}{|g_\phi'|^3}\circ g_\phi^{-1}\right)\label{Eq:FirstDeriv}\\
&=\left(\frac{\psi'}{\psi}\frac{1}{|g_\phi'|}\right)\circ g_\phi^{-1}-\left(\frac{ g_\phi''}{|g_\phi'|^2}\right)\circ g_\phi^{-1}.\nonumber
\end{align}
Recalling the bounds in Lemma \ref{Lem:BOundsGphi} one can conclude that there is a constant $C_\#>0$, depending on $H$ only, such that 
\[
\left|\frac{(L_\phi(\psi))'}{L_\phi(\psi)}\right|\le a(1+C_\#\delta)+C_\#\delta.
\]
In this computation it has been crucial that the bounds on the derivatives of $g_\phi$ are independent of $\phi$ and depend only on $H$ and $\delta$. Picking $C_1:=C_\#(1+a_1^{-1})$, the claim of the lemma follows.
\end{proof}

Combining \eqref{Eq:ContP} and \eqref{Eq:ConeCond} we get the following result.
\begin{corollary}\label{Cor:InvCones}
There are $\tilde a>0$, $\tilde \delta>0$, and $\tilde \lambda\in(0,1)$ such that for every $a>\tilde a$ and when $|\delta|<\tilde\delta$
\begin{equation}\label{Eq:ContBothOp}
PL_\phi(\psi)\in\mc V_{\tilde \lambda a}
\end{equation}
for all $\phi,\psi\in \mc V_a$. In particular $\mc T(\mc V_a)\subset \mc V_{\tilde \lambda a}$.
\end{corollary}
\begin{proof}
Pick $\tilde a=\max\{a_0,a_1\}$. Choose $\tilde \delta\in(0,\bar\delta)$ sufficiently small such that
\[
1+C_1\delta < \lambda^{-1}
\]
where $\lambda$ is as in \eqref{Eq:ContP}, and pick $\tilde\lambda:= (1+C_1\delta)\lambda$.
\end{proof}

\begin{lemma}\label{Lem:Invariance1}
There are $a,\,\alpha>0$ such that for any $\delta$ with $|\delta|<\tilde\delta$
\[
\mc T(U_{a,\alpha})\subset U_{ a, \alpha}.
\]
\end{lemma}
\begin{proof}
We already know from Corollary \ref{Cor:InvCones}, that for $a>\tilde a$ 
\[
\mc T(\mc V_a)\subset\mc V_{\tilde \lambda a}.
\]
For any $\phi\in U_{a,\alpha}$, $\alpha$ to be specified later,
\[
P\phi(x)=\sum_{i=1}^d\phi(f_i^{-1}(x))(f_i^{-1})'(x)
\]
where $\{f_i^{-1}\}_{i=1}^d$  are the inverse branches of $f$ (we assumed without loss of generality that $f'>0$), and $d$ is the degree of the map $f$. Taking first and second derivatives of this expression we obtain
\begin{align*}
(P\phi)'(x)&=\sum_{i=1}^d\phi'(f_i^{-1}(x))\cdot[(f_i^{-1})'(x)]^2+\phi(f_i^{-1}(x))\cdot(f_i^{-1})''(x)\\
(P\phi)''(x)&=\sum_{i=1}^d\phi''(f_i^{-1}(x))\cdot [(f^{-i})'(x)]^3+3\phi'(f_i^{-1}(x))\cdot (f_i^{-1})'(x)\cdot(f_i^{-1})''(x)+\\
&\quad\quad\quad+\phi(f_i^{-1}(x))\cdot(f_i^{-1})'''(x).
\end{align*}
Recall that $\sigma>1$ is  the minimal expansion of $f$, and bounding $|(f_i^{-1})''|$, $|(f_i^{-1})'''|$,$(f_i^{-1})'$ with some constant $K$ that depends on the $C^3$ norm of $f$ and on $\sigma$ only, we obtain
\begin{align*}
(P\phi)''(x)&\le \sum_{i=1}^d\alpha\phi(f_i^{-1}(x))(f_i^{-1})'(x)\sigma^{-2}+3a\phi(f_i^{-1}(x))(f^{-i})'(x)K+\\
&\quad\quad\quad+\phi(f_i^{-1}(x))(f_i^{-1})'(x)K\\
&\le (P\phi)(x)\left[\alpha\sigma^{-2}+3aK+K\right].
\end{align*}
A lower bound for $(P\phi)''(x)$ can be obtained analogously, and this proves that for $a>\tilde a$
\begin{equation}\label{Eq:InvP}
P(U_{a,\alpha})\subset U_{ \tilde \lambda a,\alpha\sigma^{-2}+3aK+K}.
\end{equation}
Similar computations show that (for $a$ sufficiently large) and with $|\delta|<\bar\delta$
\begin{equation}\label{Eq:InvmcL}
\mc L(U_{a,\alpha})\subset U_{[1+O(\delta)]a,[1+O(\delta)]\alpha+O(\delta)}
\end{equation}
where $K'$ depends on the norms of the derivatives of $H$, on $a$, and on $\bar \delta$.

Combining \eqref{Eq:InvP} and \eqref{Eq:InvmcL}, picking $\alpha$ sufficiently large and possibly decreasing $\tilde \delta$, there is $\lambda'\in(0,1)$ such that 
\[
\mc T(U_{a,\alpha})\subset U_{\tilde \lambda a,\lambda'\alpha}\subset U_{a,\alpha}.
\] 
\end{proof} 

Now that we have established the existence of an invariant cone $U_{a,\alpha}$ for $\mc T$, we find conditions on $\delta$ for the differential of $\mc T$ to preserve a cone of $\log$-Lipschitz functions, thus implying order preservation.

\begin{proposition}\label{Prop:Inclusi}
Consider $\tilde a>0$ and  $\tilde \delta>0$ as in Corollary \ref{Cor:InvCones}. Then for any $a>\tilde a$, $\alpha\ge 0$, and $\lambda''\in(\tilde\lambda,1)$ there is $\delta_1\in[0,\tilde \delta)$ -- depending on $f$, $H$, $a$, and $\alpha$ -- such that if $|\delta|<\delta_1$ 
\[
D\mc T_\phi(\mc V_a)\subset\mc V_{\lambda''a}
\]
for any $\phi\in U_{a,\alpha}$.
\end{proposition}
\begin{proof}
In the case at hand, taking into consideration Lemma \ref{Lem:Diffg} and equation \eqref{Eq:DiffmcL},  the expression of the differential of $\mc T$ at $\phi\in U_{a,\alpha}$ evaluated on $\psi\in \mc V_a$ is
\begin{equation}\label{Eq:RecDiffT}
PD\mc L_\phi(\psi)=PL_\phi\psi-\delta P\left[L_\phi\left(\frac{\phi(x)}{\int \phi}\int_{\T}H(x,y)\left[\psi-\phi\,\frac{\int\psi}{\int\phi} \right](y)dy\right)\right]'.
\end{equation}
We already know from Corollary \ref{Cor:InvCones} that $PL_\phi\psi\in\mc V_{\lambda' a}\subset \mc V_a$, and we need to find a condition on $\delta$ ensuring that the second term in the expression above is small enough so that adding it to $PL_\phi\psi$ will not push this density outside of $\mc V_a$. To this end, the  following lemma will be useful:
\begin{lemma}\label{Lem:StrictInclusion}
For every $a'<a$, every  $\psi\in \mc V_{a'}$ with $\int \psi=1$, and every $\psi_1\in C^1(\T,\R)$ with \[\|\psi_1\|_{C^1}<\min\{\frac{1}{4a'e^{a'/2}+6e^{a'}} (a-a'),\,\frac{e^{-a'/2}}{2}\},\] one has  $\psi+\psi_1\in \mc V_{a}$.
\end{lemma}
\begin{proof}
Let $\psi\in \mc V_{a'}$ and $\Delta>0$ with $\Delta e^{a'/2}<\frac{1}{2}$. Consider $\psi_1\in C^1(\T,\R)$ s.t. $\|\psi_1\|_{C^1}\le \Delta$. Using that $\frac{1}{1-y}<1+4y$ for $y\in[0,1/2]$,
\begin{align*}
\left|\frac{\psi'(x)+\psi_1'(x)}{\psi(x)+\psi_1(x)}\right|&\le\left|\frac{\psi'(x)}{\psi(x)}\right|\left|\frac{1}{1-\psi_1(x)/\psi(x)}\right|+\frac{|\psi_1'|(x)}{e^{-\frac{a'}{2}} -\Delta}\\
&\le a'[1+4\Delta e^{a'/2}]+ \Delta e^{a'/2} (1+4\Delta e^{a'/2})\\
&\le a'[1+4\Delta e^{a'/2} ] + 6\Delta e^{a'}
\end{align*}
where we also used that, for $\Delta<e^{-a'/2}$, $\psi>\psi_1$ and
\[
|\psi+\psi_1|\ge \min|\psi(x)|-\Delta\ge e^{-\frac{a'}2}-\Delta.
\]
From the assumption that $a'<a$ picking $\Delta<\frac{e^{-a'/2}}{2}$ so that 
\[
a'[1+4\Delta e^{a'/2} ] + 6\Delta e^{a'} <a,\mbox{ i.e. } \Delta < \frac{1}{4a'e^{a'/2}+6e^{a'}} (a-a')
\] is sufficient to imply $\psi+\psi_1\in\mc V_a$.
\end{proof}
Now we show that, eventually decreasing the coupling strength, the $C^1$ norm of the second term on the RHS of \eqref{Eq:RecDiffT} is bounded by $O(\delta)\|\psi\|_{C^0}$  so that Lemma \ref{Lem:StrictInclusion} can be applied to the sum in \eqref{Eq:RecDiffT} and conclude that $PD\mc L_\phi(\psi)\in \mc V_a$ for sufficiently small values of $\delta$. Let's start by computing
\begin{align}
&\left[L_\phi\left(\frac{\phi(x)}{\int \phi}\int_{\T}H(x,y)\left[\psi-\phi\,\frac{\int\psi}{\int\phi} \right](y)dy\right)\right]'=\label{Eq:ExpressionDer}\\
&\quad\quad=\left[(g_\phi^{-1}(x))'\left(\frac{\phi(g_\phi^{-1}(x))}{\int \phi}\int_{\T}H(g_\phi^{-1}(x),y)\left[\psi-\phi\,\frac{\int\psi}{\int\phi} \right](y)dy\right)\right]'\nonumber\\
&\quad\quad=(g_\phi^{-1}(x))''\,\frac{\phi(g_\phi^{-1}(x))}{\int \phi}\,\int_{\T}H(g_\phi^{-1}(x),y)\left[\psi-\phi\,\frac{\int\psi}{\int\phi} \right](y)dy+\label{Eq:ExpressionDer1}\\
&\quad\quad\quad+(g_\phi^{-1}(x))'\,\frac{[\phi(g_\phi^{-1}(x))]'}{\int \phi}\,\int_{\T}H(g_\phi^{-1}(x),y)\left[\psi-\phi\,\frac{\int\psi}{\int\phi} \right](y)dy+\label{Eq:ExpressionDer2}\\
&\quad\quad\quad+(g_\phi^{-1}(x))'\,\frac{[\phi(g_\phi^{-1}(x))]}{\int \phi}\,\int_{\T}\partial_1[H(g_\phi^{-1}(x),y)]\left[\psi-\phi\,\frac{\int\psi}{\int\phi} \right](y)dy.\label{Eq:ExpressionDer3}
\end{align}
The goal is to upper bound the $C^1$ norm of the expression in \eqref{Eq:ExpressionDer}, therefore we need to upper bound the terms in the above sum together with their first derivative. In doing so we encounter: 1) terms $(g_\phi^{-1})'$, $(g_\phi^{-1})''$, $(g_\phi^{-1})'''$, that by Lemma \ref{Lem:BOundsGphi} are all upper bounded uniformly in $\phi$ and $\delta$, provided $|\delta|<\bar\delta$; 2) terms
\begin{align*}
\frac{\phi(g_\phi^{-1}(x))}{\int \phi}&\le \frac{\phi(g_\phi^{-1}(x))}{\phi(x_{int})}\le e^{a/2}\\
\frac{[\phi(g_\phi^{-1})]'(x)}{\int \phi}&=\frac{\phi'(g_{\phi}^{-1}(x))\cdot (g_\phi^{-1})'(x)}{\int \phi}\le C_\#ae^{a/2}\\
\frac{[\phi(g_\phi^{-1})]''(x)}{\int \phi}&=\frac{\phi''(g_{\phi}^{-1}(x))[(g_\phi^{-1})'(x)]^2+\phi'(g_{\phi}^{-1}(x))(g_\phi^{-1})''(x)}{\int \phi}\le C_\#'e^{a/2}(a+\alpha)
\end{align*}
 where we defined $x_{int}$ such that $\int\phi=\phi(x_{int})$,  and used  $\phi\in U_{a,\alpha}$; and 3) terms
\begin{equation}\label{Eq:Terms}
\int_{\T}\partial_1^iH(g_\phi^{-1}(x),y)\left[\psi-\phi\,\frac{\int\psi}{\int\phi} \right](y)dy\le C_\#''\|\psi\|_{C^0}\quad\quad i\in\{0,1,2\}.
\end{equation}
All  terms in expression  \eqref{Eq:ExpressionDer} are multiplied by one of the terms in \eqref{Eq:Terms}. The same holds for an expression of their derivatives. Putting together all these considerations, one can conclude that for some $C_\#'''>0$ that depends on $a$, $\alpha$, and the norm of $H$,
\[
\left\|\left[L_\phi\left(\frac{\phi(\cdot)}{\int \phi}\int_{\T}H(\cdot,y)\left[\psi-\phi\,\frac{\int\psi}{\int\phi} \right](y)dy\right)\right]'\right\|_{C^1}\le C_\#'''\|\psi\|_{C^0}.
\]
Composing by $P$ -- i.e. by the transfer operator of a $C^2$ uniformly expanding map that acts continuously on $C^2$ densities -- at most results in multiplication of the $C^1$ norm by a factor dependent on $P$ only.

Recall that $\psi\in \mc V_a$, therefore $\|\psi\|_{C^0}\le e^{a/2}\int \psi$. This allows us to apply  lemma \ref{Lem:StrictInclusion}, and from the above estimates it follows that one can pick $|\delta|$ sufficiently small to ensure that the second term in the RHS of \eqref{Eq:RecDiffT} is small enough so that $D\mc T_\phi(\psi)=PD\mc L_\phi(\psi)\in \mc V_a$.
\end{proof}

\begin{lemma}\label{Lem:ContDiff}
Assume $|\delta|\le \bar\delta$ with $\bar\delta$ as in Lemma \ref{Lem:BOundsGphi}. For any $a,\alpha\ge 0$ and all $\phi,\psi\in U_{a,\alpha}$ and $\xi\in \mc V_a$, the map $t\mapsto D\mc T_{t\phi+(1-t)\psi}\xi$ is continuous from $[0,1]$ to $C^0(\T,\R)$.
\end{lemma}
\begin{proof}
An explicit expression for the differential can be found in equation \eqref{Eq:RecDiffT} and calling $\phi_t:=t\phi+(1-t)\psi$  and \[
\Xi_t:= \frac{\phi_t(x)}{\int \phi_t}\int_{\T}H(x,y)\left[\xi-\phi_t\,\frac{\int\xi}{\int\phi_t} \right](y)dy.
\]
we get
\begin{equation}\label{Eq:ExpDiffT}
 D\mc T_{t\phi+(1-t)\psi}\xi= PL_{\phi_t}\xi-\delta P[L_{\phi_{t}}\Xi_t]'.
\end{equation}

 By definition, we have
\[
L_{\phi_t}\xi=\frac{\xi}{g_{\phi_t}'}\circ g_{\phi_t}^{-1}.
\]
From Lemma \ref{Lem:Diffg} follows that $t\mapsto g_t$ is continuous from $[0,1]$ to $C^0(\T,\T)$, and the assumptions on $\delta$ ensure that also $t\mapsto (g_{\phi_t}')^{-1}$ and $t\mapsto g_{\phi_t}^{-1}$ are continuous. Since $\frac{\xi}{g_{\phi_t}'}$ is equicontinuous (as it is continuous and defined on a compact set), also $t\mapsto \frac{\xi}{g_{\phi_t}'}\circ g_{\phi_t}^{-1}$ is continuous from $[0,1]$ to $C^0$. Finally, since $P$ is a bounded operator from $C^0$ to $C^0$, we can conclude that $t\mapsto PL_{\phi_t}(\xi)$ is continuous. This proves that the first term in \eqref{Eq:ExpDiffT} is continuous in $t$. Before moving on, notice that if $\xi\in C^2(\T,\R)$, then one can analogously show that $t\mapsto L_{\phi_t}\xi$ is also continuous into $C^1(\T,\R
)$ as it can be easily verified by computing $(L_{\phi_t}\xi)'$.

Notice that $\Xi_t\in C^2$.
To treat the second term in \eqref{Eq:ExpDiffT}, notice that for every $t,t'$
\begin{align}
(L_{\phi_t}\Xi_t)'-(L_{\phi_{t'}}\Xi_{t'})' &= \left[(L_{\phi_t}\Xi_t)'-(L_{\phi_{t}}\Xi_{t'})'\right]+\left[(L_{\phi_t}\Xi_{t'})'-(L_{\phi_{t'}}\Xi_{t'})'\right] \label{Eq:TrIneqCont}
\end{align}
Given the regularity assumptions, $t\mapsto \Xi_t$ is continuous from $[0,1]$ to $C^1(\T,\R)$, and since $L_{\phi_t}$ is a bounded operator -- uniformly in $t$ -- from $C^1(\T,\R)$ to itself, we have that the first square bracket on the RHS of \eqref{Eq:TrIneqCont} can be made arbitrarily small in $C^0$ letting $t\rightarrow t'$.  For the second square bracket, we can use the observation made above that $t\mapsto L_{\phi_t}\xi$ is continuous to $C^1$ whenever $\xi \in C^2(\T,\R)$ and so also this term can be made arbitrarily small in $C^0$ when $t\rightarrow t'$. Again, since $P$ is bounded when acting on $C^0(\T,\R)$, also the second term in \eqref{Eq:RecDiffT} is continuous in $t$.
\end{proof}

\begin{proof}[Proof of Theorem \ref{Thm:UnifExpCoupMaps}]
By Lemma \ref{Lem:Invariance1}, there are $a$ and $\alpha$ such that: $U_{a,\alpha}$ is invariant under $\mc T$, and by Proposition \ref{Prop:Inclusi} the differential of $\mc T$ restricted to $U_{a,\alpha}$ sends $\mc V_a$ to itself and by Lemma \ref{Lem:ContDiff}, $t\mapsto D\mc T_{t\phi+(1-t)\psi}\xi$ is continuous. By Proposition \ref{Thm:OrderPreservationCondition}, the restriction of $\mc T$ to $U_{a,\alpha}$ is order preserving. Since $\mc T(U_{a,\alpha})\subset \mc V_{\lambda a}$ with $\lambda\in (0,1)$, recalling Lemma \ref{Lem:ContDiff} and \eqref{Eq:DiamLogLipFun}, it follows that $\mc T$ is a contraction.
\end{proof}

\subsection{An example of stability with  large coupling}
The work done in the previous subsection shows that if the coupling strength is sufficiently small, an STO arising from coupled uniformly expanding maps has a stable fixed point. Below we are going to show that, in specific examples, our framework allows to get explicit upper bounds for the coupling strength ensuring stability of the fixed point. In particular, we are able to deal with a class of examples where the coupling strength is so large that the STOs have multiple stable fixed points.

More precisely, let us consider the situation where the uncoupled dynamic is given by the map $f(x)=kx\mod1$ with $k>1$, and the coupling function $H$ is $C^3(\T\times\T,\R)$ and satisfies
\[
\int_\T H(x,y)dy=0\quad\quad\forall x\in \T.
\]
What follows is based on the fact that the averages of the coupling function and its derivatives with respect to the invariant measure of $f$ (i.e. the uniform Lebesgue measure) are zero.  It follows that the Lebesgue measure is invariant for the self-consistent transfer operator $\mc T$ for any value of the coupling strength.

\begin{proposition}\label{Prop:StabilityP}
Consider a system of coupled maps with $f$ and $H$ as above, and let $\mc T$ be the corresponding STO. If\footnote{We denote the $L^1$ norm by $\|\cdot\|_{1}$}
\[
|\delta|\, \max_{i=1,2}\max_{x\in\T}\|\partial_1^iH(x,\cdot)\|_1 < {2(k-1)}
\]
then there are $a_0,\alpha_0>0$ sufficiently small, such that the restriction of $\mc T$ to $\mc V_{a_0}\cap\mc C_{\alpha_0}$ is a contraction with respect to the Hilbert metric on $\mc V_{a_0}$. 
\end{proposition}
\begin{remark}
Notice that the condition on $\delta$ allows for linear transfer operators $PL_\phi$ corresponding to maps $f\circ g_{\phi}$ that are not uniformly expanding.   Uniform expansion is recovered provided that $\phi$ is sufficiently close to the constant density. As shown  below, this situation allows for the STO $\mc T$ to have multiple stable fixed points.
\end{remark}
\begin{example}\label{Ex:Example}
Consider $f(x)=5x\mod 1$ and $H(x,y)= \sin(2\pi x)\cos (2\pi y)$. Since $\max_{i=1,2}\max_{x\in\T}\|\partial_1^iH(x,\cdot)\|_1< 8\pi$, we impose that  $|\delta|< \pi^{ -1}$ so that the above theorem is satisfied and the uniform Lebesgue measure is fixed and stable (with respect to the Hilbert metric) for $\mc T$.   Notice that the Dirac measure $\delta_0$ is also fixed by $\mc T$, i.e. $\mc T\delta_0=\delta_0$, for every $\delta\in \R$, since $0$ is a fixed point of the map
\[
x\mapsto f\left(x+\delta \sin(2\pi x)\right)
\] 
for every $\delta\in \R$. Using Theorem 3.1 from \cite{ST2}, we are going to show that there is a range within $|\delta|<\pi^{-1}$ where the delta measure $\delta_0$ is a stable fixed point  in a neighborhood of measures with support contained in a sufficiently small interval around the point $0$. To apply this theorem, we need to consider
\[
G(x):= f\left(x+\delta \sin(2\pi x)\cos(2\pi x)\right)
\]
and 
\begin{align*}
g(x)&:=\partial_y\left[f(y+\delta \sin(2\pi y)\cos (2\pi x))\right]|_{y=x}\\
&=5(1+2\pi \delta\cos^2(2\pi x))
\end{align*}
and impose that: (1) $|G'(0)|<1$ and (2) $|g(0)|<1$. Theorem 3.1 of \cite{ST2} then implies that there is $\epsilon>0$ sufficiently small such that if $\mu$ is any measure on $\T$ whose support is contained in an arc containing $0$ and with diameter less than $\epsilon$, then
\[
\lim_{n\rightarrow \infty}d_W(\mc T^n\mu,\delta_0)= 0
\]
where $d_W$ is the Wasserstein metric.

Conditions (1) and (2) give the same restriction on $\delta$: 
\[
|5(1+2\pi\delta)|<1
\]
that is verified provided $\delta$ is sufficiently close to $-(2\pi)^{-1}$  and  this requirement is compatible with the condition $|\delta|<\pi^{-1}$ previously imposed. This implies the existence of a range for the coupling strength where the STO has at least two stable (in the sense specified above) fixed points: one is the uniform measure, and one is the Delta measure concentrated at zero.
\end{example}

We now turn back to  Proposition \ref{Prop:StabilityP}, whose proof will be given at the end of the section after several preliminary steps.
\begin{lemma}\label{Lem:Invariance}
There is $a_0\ge 0$, such that for every $a\in[0,a_0)$: 
\[
\mc T(\mc V_a\cap \mc C_\alpha)\subset \mc V_{\lambda a}\cap\mc C_\alpha
\]
for some $\lambda$ and $\alpha$ depending on $a$ and having the following expansions 
\[
\lambda= k^{-1}\left(1+ \frac{\delta\|\partial_1^2H\|_{1}}{2}\right) + o(1), \quad\mbox{and}\quad \alpha =a(C+o(1)),
\]
for some constant $C>0$.  Furthermore
\begin{equation}\label{Eq:Inclusion2}
PL_{\phi}\psi\in \mc V_{\lambda a}\quad\quad\forall \phi\in \mc V_a\cap\mc C_\alpha,\,\forall \psi\in \mc V_a.
\end{equation}
\end{lemma}
\begin{proof}
Without loss of generality let's assume that $\delta\ge 0$. Pick $\phi,\psi\in \mc V_a$
\begin{align*}
|(\log L_{\phi}(\psi))'|&=\left|\frac{\psi}{g_{\phi}'}\circ g_{\phi}^{-1}(x)\right|^{-1}\left|\frac{\psi'}{(g_{\phi}')^2}\circ g_{\phi}^{-1}(x)-\frac{\psi g_{\phi}''}{(g_{\phi}')^3}\circ g_{\phi}^{-1}(x)\right|\\
&\le a |(g_\phi')^{-1}|_{\infty}+\left|\frac{g_{\phi}''}{(g_\phi')^2}\right|_{\infty}.
\end{align*}
From the definition of $g_\phi$
\begin{align*}
g_{\phi}'(x)&=1-\delta\int\partial_1H(x,y)[\phi(y)-\phi_0(y)]dy\\
g_{\phi}^{(n)}(x)&=-\delta\int\partial_1^{(n)}H(x,y)[\phi(y)-\phi_0(y)]dy
\end{align*}
for $n=1,2$, where we used that $\phi_0(y):=1$ and $\int\partial_1H(x,y)dy=\int\partial_1^2H(x,y)dy=0$. By H\"older inequality
\begin{align*}
|(g_{\phi}')^{-1}|_\infty&\le \left[1-\delta\|\partial_1H\|_{1}|\phi_0-\phi|_{\infty}\right]^{-1}.
\end{align*}
Since $e^{-a/2}\le \phi \le e^{a/2}$
\begin{equation}\label{Eq:EstL1Norm}
|\phi_0-\phi|_{\infty}\le \frac{a}{2}e^{a/2},
\end{equation}
and therefore
\begin{equation}\label{eq:g'-1}
|(g_{\phi}')^{-1}|_\infty\le \left(1-\delta\|\partial_1H\|_{1}\frac{a}{2}e^{a/2}\right)^{-1}.
\end{equation}
Analogously, for $n=1, 2$
\begin{equation}\label{eq:g'n}
|(g_{\phi}^{(n)})|_\infty\le \delta\|\partial_1^{(n)}H\|_{1}\frac{a}{2}e^{a/2}.
\end{equation}
Combining all these estimates
\begin{align*}
|(\log L_{\phi}(\psi))'|&\le a \left(1-\delta\|\partial_1H\|_{1}\frac{a}{2}e^{a/2}\right)^{-1}+ \delta\|\partial_1^2H\|_{1}\frac{a}{2}e^{a/2} \left(1-\delta\|\partial_1H\|_{1}\frac{a}{2}e^{a/2}\right)^{-2}\\
&\le \left(1+ \frac{\delta\|\partial_1^2H\|_{1}}{2} +\delta o(1)\right) a,
\end{align*}
i.e. the factor multiplying $a$  can be made arbitrarily close to $ \left(1+ \frac{\delta\|\partial_1^2H\|_{1}}{2}\right)$ uniformly in $a\in[0,a_0]$ by picking $a_0$ sufficiently small.
A standard computation gives that the transfer operator for the uncoupled map $f$ satisfies $P\mc V_a\subset\mc V_{k^{-1}a}$. This implies that, as long as 
\[
k^{-1}\left(1+ \frac{\delta\|\partial_1^2H\|_{1}}{2}\right) <1,\mbox{ i.e. }\delta \|\partial_1^2H\|_1 < {2(k-1)}
\]
there is $a_0$ sufficiently small, such that for every $a\in[0,a_0]$, $\mc T\mc V_a\subset \mc V_{\lambda a}$ for some $\lambda\in (0,1)$ with 
\[
\lambda = k^{-1}\left(1+ \frac{\delta\|\partial_1^2H\|_{1}}{2}\right) + o(1). 
\] 
Which already proves \eqref{Eq:Inclusion2}.

 Now pick $\psi\in U_{a,\alpha}$,

\[
\begin{split}
( L_\phi\psi(x))''&=\left[\frac{\psi''}{(g'_\varphi)^3}-3 \frac{\psi'g_\varphi''}{(g'_\varphi)^4}-\frac{\psi g_\varphi'''}{(g'_\varphi)^4}+3\frac{\psi (g''_\varphi)^2}{(g'_\varphi)^5}\right]\circ g_\varphi^{-1}
\end{split}
\]

thus, recalling \eqref{eq:g'-1} and \eqref{eq:g'n}, we have
\begin{align*}
\frac{|( L_\phi\psi(x))''|}{L_\phi\psi}&\le \alpha |(g_{\phi}')^{-1}|_{\infty}^2+3a|g_{\phi}''|_{\infty}\cdot |(g_{\phi}')^{-1}|_{\infty}^3+\\
&+|g_{\phi}'''|_\infty\cdot |(g_{\phi}')^{-1}|_{\infty}^3+3|g_{\phi}''|^2_\infty\cdot |(g_{\phi}')^{-1}|_{\infty}^4\\
&\le \alpha+a \left[\frac{\alpha\delta\|\partial_1H\|_{1}}{2}+\delta\frac{|\partial_1^3 H|_\infty}{2}+o(1)\right].
\end{align*}

It is a quick computation to show that $P\mc C_\alpha\subset\mc C_{k^{-1}\alpha}$. Putting all the estimates together, it follows that for any $a<a_0$, there is $\alpha= O(a)$ such that 
\[
\mc T(\mc V_a\cap \mc C_{\alpha})\subset \mc V_{\lambda a}\cap \mc C_{\lambda'\alpha}
\]
for some $\lambda'\in(0,1)$.
\end{proof}
The following lemma shows that restricting the domain of $\mc T$ to $\phi$ in a cone $\mc V_a$, the deviation of $D\mc T_\phi$ from $PL_\phi$ is proportional to $\delta a$.
\begin{lemma}\label{Lem:StudyDiff}
There are $\alpha_0,\, a_0, \bar \delta>0$, there is $C>0$ such that for  $|\delta|<\bar\delta$, any $a\in [0,a_0]$, and any $\alpha\in[0,\alpha_0]$ defining
\[
\rho(x):=\left[L_\phi\left(\frac{\phi(\cdot)}{\int \phi}\int  H(\cdot,y)\left({\psi(y)}-\frac{\int \psi}{\int\phi}\phi(y)\right)dy\right)\right]'(x),
\]
 if $\phi\in U_{a,\alpha}$, $\psi\in \mc V_a$
 \[
\|\rho\|_{C^1}\le \max\{\|\partial_1H\|_1,\,\|\partial_1^2H\|_1\}\frac{a}{2}(1+o(1)).
\]
\end{lemma}
\begin{proof}
Proceeding as in the proof of Proposition \ref{Prop:Inclusi} we need to upper bound the modulus of $\rho$ and of its first derivative. An expression for $\rho$ can be found in \eqref{Eq:ExpressionDer}. We start bounding the terms appearing there to find an upper bound for $|\rho|$. Notice that every term in that sum contains a factor of the kind 
\[
\int_\T \partial_1^{(n)}H(x,y)[\phi(y)-\phi_0(y)]dy
\]
with $n\in\{0,1,2,3\}$, and, recalling \eqref{Eq:EstL1Norm}, 
\[
\left|\int_\T \partial_1^{(n)}H(x,y)[\phi(y)-\phi_0(y)]dy\right|_{\infty}\le \|\partial_1^{(n)}H\|_1 \frac{a}{2}e^{a/2}.
\]

First of all we show that  $|\rho|\le \|\partial_1^{(n)}H\|_1 \frac{a}{2}(1+o(1))$. This can be seen
bounding \eqref{Eq:ExpressionDer1} as
\begin{align*}
|(g_\phi^{-1})''|e^{a/2}C_\#\frac{a}{2}e^{a/2} \le a^2(C_{\#}+o(1))
\end{align*}
where we used 
\[
(g_{\phi}^{-1})''=\left(\frac{1}{g_\phi'(g_{\phi}^{-1})}\right)'=-\frac{g_{\phi}''(g_{\phi}^{-1})}{[g_{\phi}'(g_{\phi}^{-1})]^3}
\]
with $|g_{\phi}''|=\delta a(C_\#+o(1))$ and $|g_{\phi}'|=(1+o(1))$.
One can show that  \eqref{Eq:ExpressionDer2} is analogously a term of order $O(a^2)$ and the term in \eqref{Eq:ExpressionDer3} is bounded in absolute value by 
\[
\|\partial_1H\|_1 \frac{a}{2}[1+o(1])
\]
where one uses that $|(g_\phi^{-1}(x))'|$ and $|\frac{[\phi(g_\phi^{-1}(x))]}{\int \phi}|=1+o(a)$.
For the first derivative of $\rho$, one can bound $|\rho'|$ by showing that the derivatives of the terms \eqref{Eq:ExpressionDer1}  and \eqref{Eq:ExpressionDer2} are in absolute value of order $O(a^2)$, while the derivative of \eqref{Eq:ExpressionDer3} has leading order
\[
\|\partial_1^2H\|_1 \frac{a}{2}[1+o(1)].
\]   
\end{proof}
\begin{proof}[Proof of Proposition \ref{Prop:StabilityP}]
We apply the criterion in Proposition \ref{Thm:OrderPreservationCondition} with $U=U_{a,\alpha}$ for a suitable choice of $a$ and $\alpha$.  First of all pick $a_0$ sufficiently small so that Lemma \ref{Lem:Invariance} applies and $U$ can be showed to be invariant for every $a\in[0,a_0]$ and $\alpha=C(1+o(a))$. Notice that eventually decreasing $a_0$ and applying Lemma \ref{Lem:StudyDiff}
\begin{equation}\label{Eq:rhoC1bound}
\delta\|\rho\|_{C^1} < \frac{1}{4\lambda ae^{\lambda a}+6e^{\lambda a}} (a-\lambda a)\quad\forall a\in[0,a_0].
\end{equation}
Since $L_{\phi}\psi\in \mc V_{\lambda a}$, for every $\phi,\psi\in \mc V_a$, thanks to \eqref{Eq:rhoC1bound}, we can apply Lemma \ref{Lem:StrictInclusion} and obtain
\[
D\mc T_\phi(\psi)=L_{\phi}\psi+\delta \rho\in\mc V_a.
\]
By Proposition \ref{Thm:OrderPreservationCondition}, $\mc T$ is order preserving and a contraction on $U$.
\end{proof}

\section{ Metastable States for Coupled Maps with  Noise}

 In this section, we  give an example of noisy coupled maps with a unique stationary measure (for large finite $N$), but whose STO has a stable fixed point that is far from the unique stationary measure. This means that this stable fixed point does not approximate the asymptotic behaviour of the finite dimensional system, but it only relates to transient behaviour that can be made arbitrarily long taking larger and larger $N$.
 
\subsection{STO for Noisy Coupled Maps}

Consider $F:\T^N\rightarrow \T^N$ with $F=(F_1,...,F_N)$ as in \eqref{Eq:DetCoupledMaps}\footnote{For simplicity, as in the previous section, we consider coupled maps on $\T$, but similar results can be found for $\T^n$ with $n\in \N$. } 
and a system of noisy coupled maps
\begin{equation}\label{Eq:MarkovChainDef}
x_i(t+1)=F_i(x(t))+\eta_{i,t,F_i(x(t))} 
\end{equation}
with $x(t)=(x_1(t),...,x_N(t))\in \T^N$ and $\eta_{i,t,y}$ are i.i.d state dependent noise terms taking values on $\T$ and defined on a probability space $(\Omega,\mathbb P)$ with distribution $P(y,\cdot)$. More precisely, we mean that  \eqref{Eq:MarkovChainDef} defines a Markov chain $\{X(t)\}_{t=0}^{\infty}$ on $\T^N$, with transition probabilities $P_N(y,A)$, $y\in \T^N$ and $A\subset\T^N$, given by
\[
P_N(y,A)=\mb P(X(t+1)\in A|X(t)=y)= [P(F_1(y),\cdot)\otimes...\otimes P(F_N(y),\cdot)](A).
\]

Define the operator $M$ acting on measures
\[
(M\nu)(\cdot)=\int_{\T} P(y,\cdot)d\nu(y).
\]
The following proposition shows that the evolution of empirical distributions under \eqref{Eq:MarkovChainDef} when $N\rightarrow\infty$ is given by the STO \[
\mc T_{noisy}(\mu)=M\mc T(\mu)=M(f_\mu)_*\mu
\] which is the composition of $\mc T$, i.e. the STO associated to the deterministic coupled maps defined as in Definition \ref{Def:Self-ConsOp}, and the Markov operator $M$ associated to the noise.

\begin{proposition}\label{Prop:STOnoisyCoup}
Assume that $f$ and $H$ are Lipschitz and $y\mapsto P(y,\cdot)$ is Lipschitz from $\T$ with the Euclidean metric to the set of Borel measures with the Wasserstein distance\footnote{Defined as
\[ 
d_W(\mu,\nu)=\sup_{\substack{\phi\in \Lip_1(\T)},\,{ |\phi|_\infty\le 1}}\int_\T\phi(s)[d\nu(s)-d\mu(s)]
\]
where $\Lip_1$ denotes the set of real valued Lipschitz functions on $\T$ with Lipschitz constant less than 1.}. Suppose that $\{x_i\}_{i\in \N}$ is such that $\lim_{N\rightarrow \infty}\frac1N\sum_{j=1}^\infty \delta_{x_j}=\mu\in \mc M_1(\T)$ weakly, then almost surely and weakly
\[
\lim_{N\rightarrow \infty}\frac1N\sum_{j=1}^N\delta_{F_j(x)+\eta_{i,t,F_j(x)}}=M(f_\mu)_*\mu
\]
where $x=(x_1,...,x_N)$ and $f_\mu=f\circ g_\mu$ as in Definition \ref{Def:Self-ConsOp}.
\end{proposition}
\begin{proof}
Let $\bo x:=\{x_n\}_{n\in \N}$ be fixed and satisfy the assumptions.  Pick any $1$-Lipschitz function $\phi$ such that $|\phi|_\infty<1$. Let's make explicit for a moment the dependence on $N$ of the notation and write: $F^{(N)}=(F_1^{(N)},...,F_N^{(N)}):\T^N\rightarrow \T^N$ in place of $F$, and $x^N=(x_1,...,x_N)$. Define the triangular array of random variables $\{Z_{j}^N\}_{N=1, j=1}^{\infty,N}$ on the probability space $(\T^\N\otimes \Omega, \,\Leb_\T^{\otimes \N}\otimes \mathbb P)$ with $Z_{j}^N:\T^\N\times \Omega \rightarrow \R$
\[
Z^N_j:=\phi\left(F^{(N)}_j(x^N)+\eta_{i,t,F^{(N)}_j(x^N)}\right)
\]
where for $(\bo x,\omega)\in \T^\N\times\Omega$, $x^N$ is the projection of $\bo x$ on the first $N$ coordinates.  These random variables are independent and have finite moments of all orders. One can therefore apply the strong law of large numbers for triangular arrays and obtain that $\Leb_\T^{\otimes \N}\otimes \mathbb P$-a.s.
\begin{equation}\label{Eq:LLNTrArrays}
\lim_{N\rightarrow\infty}\frac1N\sum_{j=1}^N\left(Z^N_j-\mb E[Z^N_j]\right)=0.
\end{equation}

From the assumption on $\bo x$, for any $\epsilon>0$ there is $N_0\in \N$  such that for any $N\ge N_0$, $d_W(\frac1N\sum_{j=1}^N \delta_{x_j},\mu)<\|H\|_{C^1}^{-1}\epsilon$; this ensures that  for every $N\ge N_0$ and any $j\in\{1,...,N\}$,
 $|F^{(N)}_j(x^N)-f_\mu(x_j)|<\epsilon$ and
 \begin{align*}
 \frac1N\sum_{j=1}^N\mb E[Z^N_j]&=\frac1N\sum_{j=1}^N\int \phi(F^{(N)}_j(x^N)+\eta)P(F^{(N)}_j(x^N),d\eta) \\
 &=\frac1N\sum_{j=1}^N\int \phi(f_\mu(x_j)+\eta)P(f_\mu(x_j),d\eta)+O(\epsilon)
 \end{align*}
 where in the last estimate we used the regularity properties of $\phi$ and $y\mapsto P(y,\cdot)$.
Now,
\[
\frac1N\sum_{j=1}^N\int \phi(f_\mu(x_j)+\eta)P(f_\mu(x_j),d\eta)=\int\phi(s)\,d\left( M\frac{1}{N}\sum_{j=1}^N\delta_{f_\mu(x_j)}\right)(s).
\]
By regularity of $f$ and $H$, recalling the definition of $f_\mu$,
\[
\frac{1}{N}\sum_{j=1}^N\delta_{f_\mu(x_j)}=\frac{1}{N}\sum_{j=1}^N(f_\mu)_*\delta_{x_j}\rightarrow (f_\mu)_*\mu
\]
and the continuity of $M$ implies
\[
\frac{1}{N}\sum_{j=1}^N\delta_{f_\mu(x_j)}\rightarrow M(f_\mu)_*\mu.
\]

Putting all these estimates together, for all $\epsilon>0$ there is $N_0'$ such that for every $N\ge N_0'$
\[
\frac1N\sum_{j=1}^N\mb E[Z^N_j]=\int \phi(s)d(M(f_\mu)_*\mu)(s)+O(\epsilon)
\]
which together with \eqref{Eq:LLNTrArrays} implies
\[
\lim_{N\rightarrow \infty} \frac{1}{N}\sum_{j=1}^N\phi\left(F^{(N)}_j(x^N)+\eta_{i,t,F^{(N)}_j(x^N)}\right)=\int \phi(s)d(M(f_\mu)_*\mu)(s)
\]
and this implies the statement of the proposition. 

\end{proof}

\subsection{An Example with Metastable States}

\paragraph{Deterministic Evolution} 
Consider the (deterministic) system of coupled maps  from Example \ref{Ex:Example} where recall that  the  coordinates are $x(t)=(x_1(t),...,x_N(t))\in\T^N$ and their time evolution is prescribed by 
\begin{equation}\label{Eq:NoisyCoupledSystem}
    x_i(t+1)=F_i(x(t)):=k\left( x_{i}(t)+\delta\sin(2\pi x_i(t))\frac{1}{N} \sum_{j=1}^N\cos(2\pi x_j(t))\right)\mod 1
\end{equation}
where in the aforementioned example $k=5$.

 For certain values of $\delta$, one can show that $0\in \T^n$ is an attracting fixed point with a prescribed rate of contraction.
\begin{lemma}\label{Lem:BasinFixedPoint}
For $\delta\in\left(-\frac{2+3k}{6\pi k},-\frac{3k-2}{6\pi k}\right)$ and $N$ large enough, $F:\T^N\rightarrow \T^N$ has a uniformly attracting fixed point at $0\in\T^N$. Furthermore, there is $\Delta>0$\footnote{Here, and from now on, $\Delta>0$ is always assumed to be small enough so that $[-\Delta,\Delta]$ is a small arc around 0.}, such that the set $B_\Delta=[-\Delta,\Delta]^N\subset\T^N$ is contained in the basin of attraction of $0$, and for every $i$ and $x=(x_1,...,x_N)\in B_\Delta$
\begin{equation}\label{Eq:Cont}
d_\T(F_i(x),0)\le \frac23 d_\T (x_i,0). 
\end{equation}
\end{lemma}
 \begin{proof}
\[
\partial_{x_i}F_i(x_1,...,x_n)=k\left(1+\delta 2\pi \cos(2\pi x_i)\left[\frac{1}{N} \sum_{j=1}^N\cos(2\pi x_j)\right]\right)-\frac{k 2\pi \delta}{N}\sin(2\pi x_i)^2
\] 
With the choice 
\[
\delta\in\left(-\frac{2+3k}{6\pi k},-\frac{3k-2}{6\pi k}\right)
\]
one has $k(1+2\pi\delta)<2/3$, and  there is $\Delta>0$ small enough such that 
\[
k(1+\delta 2 \pi \cos(2\pi \Delta)\cos(2\pi \Delta))< \frac23.
\]
 For any $x\in B_\Delta=[-\Delta,\Delta]^N$ and $N$ large enough
\[
|\partial_{x_i}F_i(x)|\le \frac23.
\]
This and the fact that 
\[F_i(x_1,..,x_{i-1},0,x_{i+1},...,x_N)=0
\] for every $x_1,\,...,x_{i-1},\,x_{i+1},...,x_N$ implies that $0$ is an attracting fixed point and that \eqref{Eq:Cont} is satisfied. 
 \end{proof}

\begin{remark}
For $k=5$, the values of $\delta$ in the lemma above satisfy $|\delta|< \pi^{-1}$. In this range for $\delta$, recalling the computations in Example \ref{Ex:Example}, the STO  has the Lebesgue measure as a stable fixed point. 
\end{remark}

\paragraph{ Random Evolution} We now consider the random discrete time equations 
\[
x_i(t+1)= F_i(x(t))+ \eta_{i,t,F_i(x(t))}
\]
 with $F_i$ as above and we will show that for a suitable choice for the noise, for every large enough but finite $N$, the coupled noisy system gives rise to a geometric Markov chain with 
a unique stationary measure supported on $B_\Delta$, but the corresponding STO has at least one stable fixed point that is unrelated to the $N\rightarrow \infty$ limit of this measure.

 Let's call $P(x,\cdot)$ the transition probability associated to this noise, i.e. 
 \[
 P(x,A)=\mathbb P(\eta_{i,t,x}\in A)
 \] 
 and $M$ the associated generator:
 \[
 M \nu(\cdot):=\int P(x,\cdot)d\nu(x).
 \]
 Then, as showed above, the STO for this noisy system is just the composition of $\mc T$, the STO associated to the deterministic part, and the generator for the noise: 
 \[
\mc T_{noisy}\nu=  M\mc T\nu .
 \]

 Before making explicit the expression for $P(x,\cdot)$, let's first choose a smooth density  $\tilde \xi:\T\rightarrow\R$ that is supported on $[-\Delta/3,\Delta/3]$,  $\Delta$  as in Lemma \ref{Lem:BasinFixedPoint}, and that is symmetric around $0$. For example, we can take the normalised bump function
 \[
\tilde \xi(z)= \left\{ \begin{array}{ll}c_{\tilde \xi}\exp\left[-(\Delta^2/9-z^2)^{-1}\right] & z\in(-\Delta/3,\Delta/3)\\
0 & \mbox{otherwise}
\end{array}
\right.
 \]
  where $c_{\tilde \xi}>0$ is a normalising constant.
 
 Now let's define the properties that $ P(x, \cdot)$ should satisfy;  calling $\xi(x,\cdot)$ its density, we impose that: 
 
 a.1) $ \xi(x,y)=\tilde \xi(x-y)$, with $\tilde \xi$ as above, for $x\in [-\frac{2}{3}\Delta,\frac{2}{3}\Delta]$;
 
  \noindent and that: there is $\gamma>0$ such that for every $x\in\T$: 
 
 a.2) $\int_{[-\Delta,\Delta]} \xi(x,y)dy>\gamma$;
 
 a.3) $\| \xi(\cdot,y)-\tilde \xi(\cdot,y)\|_{C^2}=O(\gamma)$ for every $y\in \T$, where we denoted $\tilde \xi(x,y):=\tilde \xi(x-y)$.


Condition a.1) together with \eqref{Eq:Cont} and the fact that $\tilde \xi$ is supported on $[-\Delta/3,\Delta/3]$ ensures that a random orbit that enters $[-\Delta,\Delta]^N$ will be trapped there forever.
Condition a.2) ensures that every random orbit has probability larger than $\gamma^N>0$ of ending in  $[-\Delta,\Delta]^N$ in one step.
 Condition a.3) ensures that the two kernels are close for $\gamma$ small, and so are their associated generators; this will be instrumental in showing that the STO has a stable fixed point close to the uniform Lebesgue measure. 

\begin{example} Fix a smooth bump function
$\iota(x)$ that takes values on $[0,1]$,  is equal to zero for $x\in[-2\Delta/3,2\Delta/3]$, and is equal to one on $x\in [-2\Delta/3,2\Delta/3]^c$. 
Then define
\[
\xi(x,y)= \left(1-\frac{\gamma\iota(x)}{2\Delta}\right)\tilde\xi(x-y)+\frac{\gamma\iota(x)}{2\Delta}.
\]
With this choice, $\xi(x,\cdot)$ is a linear interpolation between $\tilde \xi(x-\cdot)$ and the uniform density. This  immediately ensures that a.1) and a.2) are satisfied. Also a.3) can be easily verified noticing that
\[
\xi(x,y)=\tilde\xi(x-y)+\gamma\left[\frac{\iota(x)}{2\Delta}-\frac{\iota(x)}{2\Delta}\tilde\xi(x-y)\right]
\]
and that the $C^2$ norm of the expression inside the square bracket does not depend on $\gamma$.
\end{example}

\begin{proposition}\label{Prop:ExMetast}
For  $\delta\in\left(-\frac{2+3k}{6\pi k},-\frac{3k-2}{6\pi k}\right)$ the two following facts hold:
\begin{itemize}
\item[i)] For  $N\in \N$ large enough, the Markov chain prescribed by the noisy coupled maps in \eqref{Eq:NoisyCoupledSystem}  is geometrically ergodic and its unique stationary measure is supported on $[-\Delta,\Delta]^N$;
\item[ii)] For $\gamma$ (as in a.2 and a.3 above) sufficiently small the STO associated with the noisy coupled maps \eqref{Eq:NoisyCoupledSystem} has a stable fixed measure with $\log$-Lipschitz density; for $\gamma\rightarrow 0$ its density tends to the Lebesgue measure in the $\infty$-topology.
\end{itemize}
\end{proposition}
\begin{remark}
    Notice that the set $[-\Delta,\Delta]^N$ has exponentially small Lebesgue measure; it is therefore ``incompatible" with the stable fixed point for the STO found in point ii) that would prescribe that the coordinates $\{x_i(t)\}_{i\in\N}$ are roughly uniformly distributed.
\end{remark}
\begin{proof}[Proof of Proposition \ref{Prop:ExMetast}]
Point i) is a straightforward consequence of the assumptions and results on geometrically ergodic Markov chains (see e.g. \cite{MT}). Recall that $P_N(x,\cdot)$, where $x\in \T^N$,  is the transition probability of the full noisy coupled system. Denoting by $P^n_N$ the $n$-th fold convolution of the transition kernel, it is not hard to show that given assumptions a.1) and a.2) there exist $n\in \N$ and $\epsilon>0$ such that for every $x\in \T^N$
\[
 P^n_N(x,\cdot )\ge \epsilon \Leb_{[-\Delta/6,\Delta/6]^N}(\cdot)
\]
where $n>\frac{\log(1/6)}{\log(2/3)}$, so that $(2/3)^n<1/6$. This implies that the Markov chain is geometrically ergodic. That the unique stationary measure is supported on $[-\Delta,\Delta]^N$ follows from the fact that this set is forward invariant (from assumption a.1 on the kernel and \eqref{Eq:Cont}) and the fact that every random orbit enters this set with positive probability, and therefore a.e. orbit enters this set in finite time.

Point ii). We start by characterising the action of $M$ on densities with $\log$-Lipschitz densities. Take any $0<a'<a$ and $\phi\in \mc V_{a'}$. First of all notice that defining
\[
\tilde M\phi(x)=\int \tilde \xi(y)\phi(y-x)dy,
\]
 it is easy to check that $\tilde M\phi\in \mc V_{a'}$, and therefore $\tilde M\mc V_{a'}\subset \mc V_{a'}$. Analogously, one can show that $ \tilde M \mc C_\alpha\subset \mc C_\alpha$.
Now, by a.3, the operator
\[
( M\phi)(y)=\int \xi(x,y)\phi(x)dx
\]
satisfies
\begin{align*}
    \| M\phi-\tilde M \phi\|_{C^1} =\left\|\int \left(\xi_\cdot(y)-\tilde \xi(\cdot-y)\right)\phi(y)dy\right\|_{C^1}\le O(\gamma) \int\phi(y)dy.
\end{align*}
Since $ M\phi= \tilde M\phi +( M\phi-\tilde M\phi)$, applying Lemma \ref{Lem:StrictInclusion} we obtain that there is $\gamma$ small enough such that 
\[
 M\mc V_{a'}\subset \mc V_a.
\]
Analogously, using again a.3), for every $\alpha'<\alpha$, there is $\gamma$ small enough such that $ M\mc C_{\alpha'}\subset\mc C_\alpha$.

Now, since from the result on the deterministic part (Lemma \ref{Lem:Invariance1} and Proposition \ref{Prop:StabilityP} and their proofs) we know that for $\delta$ in the prescribed range 
\[
\mc T U_{a,\alpha}\subset U_{a',\alpha'}\quad\quad D\mc T|_{U_{a,\alpha}}(\mc V_a)\subset \mc V_{a'}
\]
for some $a'<a$ and $\alpha'<\alpha$, it follows that for the noisy STO, provided $\gamma$ is small enough
\[
\mc T_{noisy} U_{a,\alpha}= M\mc T U_{a,\alpha}\subset U_{a,\alpha}\quad\quad D\mc T_{noisy}|_{U_{a,\alpha}}(\mc V_a)= MD\mc T|_{U_{a,\alpha}}(\mc V_a)\subset \mc V_{a}
\]
Proposition \ref{Thm:OrderPreservationCondition} implies that  $\mc T_{noisy}$ is order preserving and, arguing as in the proof of Proposition \ref{Prop:StabilityP}, it has an attracting fixed point in $U_{a,\alpha}$.

That the density tends to the uniform one for $\gamma\rightarrow 0$, follows from the fact that the stable fixed point for $\tilde M\mc T$ with $\log$-Lipschitz  density is Lebesgue, and condition a.3) on the noise kernel and statistical stability for Markov chains imply that the fixed point for $ M\mc T$ must get arbitrarily close to Lebesgue when $\gamma\rightarrow 0$.
\end{proof}

\noindent \textbf{Acknowledgements}

 S.G. acknowledges the MIUR Excellence Department Project awarded to the
Department of Mathematics, University of Pisa, CUP I57G22000700001. S.G. was
partially supported by the research project "Stochastic properties of
dynamical systems" (PRIN 2022NTKXCX) of the Italian Ministry of Education
and Research. M.T. acknowledges Marie Slodowska-Curie Actions: "Ergodic
Theory of Complex Systems", project no. 843880. The authors acknowledge the UMI Group “DinAmicI” ({www.dinamici.org}) and the INdAM group GNFM.
\\
\textbf{Declarations}

The authors have no relevant financial or non-financial interests to disclose. 
Data sharing not applicable to this article as no datasets were generated or analysed during the current study.

\end{document}